\documentclass{elsarticle}
\usepackage{imakeidx}
\makeindex[intoc]

\newcommand{\IN}[1]{\index{#1|BH}#1}
\usepackage[utf8]{inputenc}
\usepackage[margin=1in]{geometry} 
\usepackage{amsmath,amsthm,amssymb}
\usepackage{bbm}
\usepackage{dsfont}
\usepackage{setspace}
\usepackage{xcolor}
\usepackage{graphicx}
\usepackage{tikz}
\usetikzlibrary{shapes.geometric, arrows}
\usepackage[colorlinks]{hyperref}
\hypersetup{
    colorlinks=true,
    linkcolor=blue,
    filecolor=magenta,      
    urlcolor=cyan,
    citecolor=red
}
\usepackage{cleveref}
\usepackage{titleref}
\allowdisplaybreaks

\def\R{{\mathbb R}} 
\def\N{{\mathbb N}} 
\def\G{{\mathcal G}} 
\def\W{{\mathcal W}} 

\newtheorem*{lemma*}{Lemma}

\theoremstyle{plain}
\newtheorem{theorem}{Theorem}[section]
\newtheorem{proposition}[theorem]{Proposition}
\newtheorem{lemma}[theorem]{Lemma}

\theoremstyle{definition}
\newtheorem{definition}[theorem]{Definition}

\theoremstyle{remark}

\numberwithin{subcase}{case}

\begin{document}
\begin{frontmatter}

\title{A Partial Characterization of Robinsonian $L^p$ Graphons}

\author{Teddy Mishura}
\ead{tmishura@torontomu.ca}
\address{Department of Mathematics, Kerr Hall South, Toronto Metropolitan University in Toronto, ON}

\begin{abstract}
We present a characterization of Robinsonian $L^p$ graphons for $p > 5$. Each $L^p$ graphon $w$ is the limit object of a sequence of edge density-normalized simple graphs $\{G_n/\|G_n\|_1\}$ under the cut distance $\delta_{\Box}$. A graphon $w$ is Robinson if it satisfies the Robinson property: if $x\leq y\leq z$, then $w(x,z)\leq \min\{w(x,y),w(y,z)\}$, and it is Robinsonian if $\delta_{\Box}(w,u)=0$ for some Robinson $u$. In previous work, the author and collaborators introduced a graphon parameter $\Lambda$ that recognizes the Robinson property, where $\Lambda(w) = 0$ precisely when $w$ is Robinson. Using functional analytic arguments, we show here that for $p > 5$, the Robinsonian $L^p$ graphons $w$ are precisely those that are the cut distance limit object of graphs $G_n$ such that $\Lambda(G_n/\|G_n\|_1) \to 0$. 
\end{abstract}

\end{frontmatter}
\section{Introduction}
The analysis of structure in large networks has been an important question since the beginnings of computer science---many real world data structures are modeled as vast, complex graphs, and determining both local and global behavior is more desirable now than ever. In particular, many physical networks are dynamic in nature, growing in size and structure as new data is collected; equally as important in such changing systems is recognizing and characterizing emergent properties such as clustering or connectedness. The language of graph limit theory offers a powerful method to attack these problems in the form of $L^p$ \textit{graphons}, symmetric measurable functions from $[0,1]^2 \to \R$ with finite $p$-norm $\|\cdot\|_p$. 

First introduced by Lov\'asz and Szegedy \cite{LOVASZ2006933}, graphons are symmetric measurable functions from $[0,1]^2 \to [0,1]$, and can be viewed as the limit objects of sequences of graphs under the \textit{cut norm} $\|\cdot\|_{\Box}$, where
$$\|w\|_{\square}=\sup_{A,B \subseteq [0,1]}\Bigg|\iint_{A\times B} w(x,y)dxdy\Bigg|.$$ 
For any labeled graph $G$, the adjacency matrix $A_G$ can be inserted into the unit square as a $\{0,1\}$--valued step function $w_G$; graphons form the completion of these functions under $\|\cdot\|_{\Box},$ and we denote the space of graphons by $\W_0$. Given an \textit{unlabeled} graph $H$, where the adjacency matrix is ill defined, we compare its distance to a graphon using the \textit{cut distance} $\delta_{\Box}$---the minimum difference in cut norm over all labelings of $H$. This embedding of the adjacency matrix forces sequences of graphs $\{G_n\}$ with vanishing edge density to converge to the zero graphon, yielding a trivial result for many classes of networks, such as any graph family whose edges are governed by a power law. Borgs et~al. resolved this issue in \cite{Borgs_2018,Borgs_2019} by instead considering the convergence of the graph sequence $\{G_n\}$ normalized by its edge density $\{G_n / \|G_n\|_1\}$. The limit functions of these sequences need only be bounded in $p$-norm and are thus commonly known as $L^p$ \textit{graphons}, whose space is denoted $\W^p.$

The main attraction of graph limit theory is the ability to answer combinatorial problems using analytic tools. Thus, when considering the problem of recognizing emergent structure and behaviors in a collection of large networks $\mathcal{G}$, it is tempting to find an $L^p$ graphon $w$ that is extremely close to $\mathcal{G}$ in cut norm and then test $w$ for the desired structure or behavior. While this approach is often quite powerful, it is not always the case that a combinatorial property of graphs translates to graphons (see for example \cite{HLADKY2020103108}).  Thus, recognizing and characterizing the properties that do extend from graphs to graphons is very useful for the analysis of growing networks, as determining the underlying graphon of a collection of networks can give insight into the collection as a whole. 

We study here the Robinson property: an $L^p$ graphon $w$ is Robinson if for all $x \leq y \leq z$, then
\begin{equation*}
    w(x,z) \leq \min(w(x,y),w(y,z)).
\end{equation*}
These functions were defined as \textit{diagonally increasing} in \cite{Chuangpishit_2015} because they increase toward the main diagonal along a horizontal or vertical line, and are generalizations of Robinson matrices.  Sometimes called R-matrices, Robinson matrices appear in the classic problem of seriation (see \cite{Liiv_2010} for a comprehensive review of seriation) and their study is a field of much interest \cite{chepoi2009,flammarion2016optimal,fogel2014,laurent2017}. Given a sequence of graphs $\{G_n\}$, can one recognize if they are sampled from a Robinson graphon? It is simpler to tackle the labeled case first and then extend to the unlabeled case, so the problem becomes the following: Construct a function $\Phi : \W^p \to [0,\infty)$ that \textit{measures} the Robinson property of a graphon. Such functions are called \textit{graphon parameters}, and can be chosen to recognize desired structure in $w$. A graphon parameter $\Phi$ is suitable for Robinson measurement if it is subadditive and satisfies the following three properties:
\begin{itemize}
    \item\textbf{(Recognition)} $\Phi(w) = 0$ if and only if $w$ is Robinson a.e.

    \item\textbf{(Continuity)} $\Phi$ is continuous with respect to the cut norm.

    \item\textbf{(Stability)} Given a graphon $w$, there exists a Robinson graphon $u$ such that $\|w-u\|_{\Box} \leq f(\Phi(w))$, where $f:[0,\infty) \to [0,\infty)$ is a nondecreasing function with $\lim_{x \to 0^+} f(x) = 0.$ 
\end{itemize}

Provided a sequence of associated graphons $\{w_{G_n}\}$ and a graphon $w \in \W^p$ such that $\|w_{G_n}-w\|_{\Box}\to0$, recognition and stability combined with subadditivity guarantee that the Robinson graphon $u$ is close in cut norm to each $w_{G_n}$ by some measure of how Robinson $w$ is, while continuity ensures that the if the $\{w_{G_n}\}$ are Robinson, then so is $w$. Thus, any such $\Phi$ will be able to recognize when a sequence of \textit{labeled} graphs is sampled from a Robinson graphon or a graphon that is \textit{almost} Robinson (see \cite{Ghandehari2023RobustRO}). However, extending these results to \textit{unlabeled} graphs is both difficult and dependent on the density of the graph sequence in question. 

In this paper, we resolve this problem for unlabeled graphs that converge to $L^p$ graphons for $p>5$ in terms of the graphon parameter $\Lambda$ defined in \cite{Ghandehari2023RobustRO}. All graphons in $\W^p$ for $p > 5$ are ``$\Lambda$-close'' to certain Robinson graphons in the cut norm---thus, if a graph sequence $\{G_n\}$ whose associated graphons $\{w_{G_n}\}$ grow increasingly Robinson converges to a graphon $w$ in cut distance, there must exist a sequence of Robinson graphons that converge to $w$ in cut distance. Through a combination of the specific properties of these Robinson approximations and functional analytic arguments, we showed that in this case there must exist some Robinson graphon $u$ such that $\delta_{\Box}(u/\|u\|_1,w/\|w\|_1) = 0$. Therefore, $w$ is Robinsonian precisely when it is the limit object of a sequence of simple graphs $\{G_n\}$ such that $\Lambda(G_n/\|G_n\|_1) \to 0$. 

\section{Notation and background}\label{sec:notationbg}
In this section, we define basic notation used throughout the paper; afterwards, we present graph limit theory, stating required definitions alongside notable results, then transition into explanation of the Robinson property for both matrices and graphons.

All functions and sets discussed during this paper are assumed to be measurable. The Lebesgue measure of a set $A \subseteq \R$ is denoted by $|A|$. The notation $\|\cdot\|_p$ is used to represent the standard $p$-norm; that is, for a function $f:X \to \R$ and $1 \leq p < \infty$,
\begin{equation*}
    \|f\|_p := \left(\int_X |f|^p\right)^{\frac{1}{p}} \mbox{ and } \|f\|_{\infty} := \inf_{a \in \R}\{f(x) \leq a \mbox{ a.e.}\}.
\end{equation*}
Let $X$ be a normed vector space with inner product $\langle\cdot,\cdot\rangle,$ and let $X^* = \{\phi:X \to \R~|~\phi \text{ continuous}\}$ be the \textit{dual} of $X.$ If $\{\mu_n\}_{n \geq 1} \subset X^*$ and $\mu \in X^*$, we say that $\mu_n$ converges \textit{weak}$^*$ to $\mu$, writing $\mu_n \to \mu$, if for all $x \in X$, we have that
\begin{equation*}
    \langle x,\mu_n \rangle \to \langle x,\mu\rangle.
\end{equation*}
In particular, the function space $L^p[0,1]^2$ has $L^q[0,1]^2$ as its dual, where $q = p/(p-1)$. Thus, for $f_n$ to converge weak$^*$ to $f$ in $L^q[0,1]^2$, for all $g \in L^p[0,1]^2$, it must be true that
\begin{equation*}
    \left(\iint_{[0,1]^2}|f_n|^p|g|^pdxdy\right)^{\frac{1}{p}} \to \left(\iint_{[0,1]^2}|f|^p|g|^pdxdy\right)^{\frac{1}{p}}.
\end{equation*}
\subsection{Graph limit theory}
A \textit{graph} $G = (V,E)$ is a collection of vertices $V(G)$ and edges $E(G)$. In this paper, we will only consider \textit{simple} graphs: those with neither loops nor multiple edges between vertices. We also draw a distinction between \textit{labeled} and \textit{unlabeled} graphs, where labeled graphs $G$ have the vertex set $V(G) = \{1,\ldots,|V(G)|\}$. Any unlabeled graph $H$ can be labeled by uniquely assigning elements from $\{1,\ldots,|V(H)|\}$ to $V(H)$. 

Every labeled graph $G$ has an \textit{adjacency matrix} $A_G$ with $(A_G)_{ij} = 1$ if $i \sim j$ and 0 otherwise, where $i \sim j$ denotes that vertices $i$ and $j$ are connected by an edge. Furthermore, $A_G$ can be embedded into $[0,1]^2$ as a step function $w_G$ as follows: Divide $[0,1]$ into $n$ intervals $I_1,...,I_n$ of equal measure, and for $x \in I_i$ and $y \in I_j$, let $w_G(x,y) = (A_G)_{ij}$. Importantly, we note that both $A_G$ and $w_G$ are dependent on the labeling of $V(G)$: see Figure \ref{fig:weightedgraph} for an example. Unless otherwise noted, we always consider a labeled graph $G$ the same as its associated step function $w_G$. These functions $w_G$ are members of a broader class of functions known as \textit{graphons}, symmetric functions $w: [0,1]^2 \to [0,1]$. The space of graphons is denoted $\W_0$. 
\begin{figure}[t]
    \centering
    \includegraphics[width=0.7\textwidth]{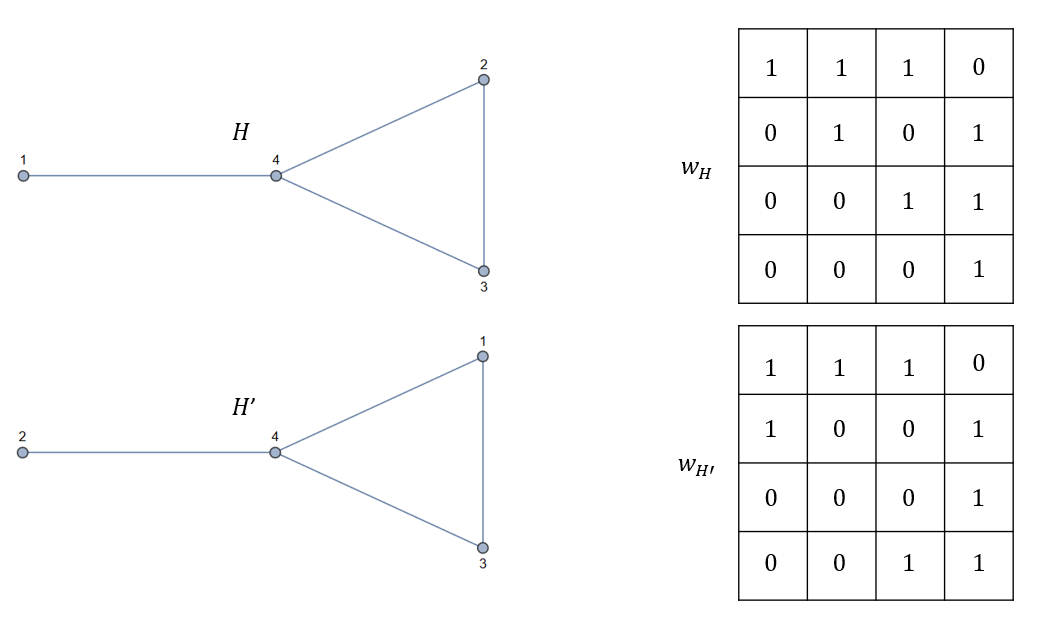}
    \caption{Two isomorphic labeled graphs $H$ and $H'$ with different associated \IN{step function}s. Note that $w_H$ and $w_{H'}$ are symmetric as functions on $[0,1]^2,$ not as matrices.}
    \label{fig:weightedgraph}
\end{figure}
The foundational idea of graph limit theory is that a sequence of graphs $\{G_n\}$ can converge to a graphon $w$. We shall view this entirely through an analytic framework, where we consider the convergence of the functions $\{w_{G_n}\}$ to the function $w$ in the famous \textit{cut norm} $\|\cdot\|_{\Box}$, first defined by Frieze and Kannan \cite{friezekannan}. The cut norm of $w:[0,1]^2 \to \R$ is given by
\begin{equation*}
    \|w\|_{\Box} = \sup_{S,T \subseteq [0,1]}\Bigg|\iint_{S \times T}w(x,y)dxdy\Bigg|.
\end{equation*}
For a labeled graph $G$, we define $\|G\|_{\Box} := \|w_G\|_{\Box}$. A sequence of labeled graphs $\{G_n\}$ \textit{converges} to a graphon $w$ if $\|w_{G_n}-w\|_{\Box}\to0$.

For a sequence of unlabeled graphs $\{H_n\}$ to converge to a graphon $w$, something more is needed: $\|H_n-w\|_{\Box}$ is only defined if $H_n$ is labeled. Thus, the natural extension is to find the minimum cut norm difference over all possible labelings of $H_n$. The \emph{cut distance} generalizes this idea to functions on $[0,1]^2$, where it is now the set $[0,1]$ that must be relabeled. Let $\Phi$ be the set of all bijections $\varphi:[0,1]\to[0,1]$ such that if $A \subset [0,1]$, then $|\varphi(A)| = |A|$. We call such bijections \textit{measure preserving}. The cut distance between two functions $u, w:[0,1]^2\to \R$ is given by
\begin{equation*}
    \delta_{\square}(u,w) = \inf_{\varphi \in \Phi}\|u-w^{\varphi}\|_{\square},
\end{equation*}
where $w^{\varphi}(x,y) = w(\varphi(x),\varphi(y)).$ However, the distance $\delta_{\square}$ is only a pseudometric---two nonidentical graphons can have cut distance 0. We fix this by identifying graphons of cut distance 0 from each other to get the set $\widetilde{\W}_0$ of \textit{unlabeled graphons}. It is a classic result of graph limit theory that the metric space $(\widetilde{\W}_0,\delta_{\Box})$ is compact. Unless otherwise stated, the convergence of a graph sequence $\{G_n\}$ refers to the convergence of the unlabeled graphons $\{w_{G_n}\}$ in $\delta_{\Box}$. Furthermore, it is clear that for any graphon $w \in \W_0$, 
\begin{equation*}
    \|w\|_{\square} \leq \|w\|_{1} \leq \|w\|_{p} \leq \|w\|_{\infty}.
\end{equation*}

Thus, if $\|w_{G_n}\|_1 \to 0$ for some sequence of graphs $\{G_n\}$, that sequence must trivially converge to 0. These sequences are called \textit{sparse}; graph sequences that are not sparse are called \textit{dense}. We note that $\|w_{G}\|_1 = |E(G)|/|V(G)|^2$ is the \textit{edge density} of a graph, and so sparse graph sequences are exactly those with a subquadratic number of edges. Such graph sequences were shown in \cite{Borgs_2018} to have nontrivial limits when normalized by their edge density---or, for the associated \IN{graphon}s, normalized by their $L^1$ norm. Thus, for sparse graph sequences $\{G_n\}$ to converge, there must be some limit graphon $w$ such that
\begin{equation*}
\left\|\frac{w_{G_n}}{\|w_{G_n}\|_1}-w\right\|_{\Box} \to 0 ~\text{ or }~ \delta_{\Box}\left(\frac{w_{G_n}}{\|w_{G_n}\|_1},w\right) \to 0
\end{equation*}
as opposed to $\{w_{G_n}\}$. However, the limit objects of such sequences must be unbounded, and so we introduce $L^p$ \textit{graphons}: symmetric, measurable functions $w \in L^p([0,1]^2)$. We refer to the space of \IN{$L^p$ graphon}s as $\W^p$, noting that $\W_0 \subset \W^{\infty} \subseteq \W^p$ for $p \geq 1$. This generalization of the notion of a graphon still brings with it powerful results from the dense theory: The metric space $(\widetilde{\W}^p,\delta_{\Box})$ has a compact unit ball.

As every graph can construct a graphon, so can every graphon construct a graph; \textit{sampling} from a graphon refers to creating one of these random graphs. Given a \IN{graphon} $w$ and a set $S = \{x_1,...,x_n\}$ where each $x_i \in [0,1]$, create a random simple \IN{graph} $\G(n,w)$ on $V(G) = \{1,\ldots,n\}$ as follows: Connect vertices $i$ and $j$ with probability $w(x_i,x_j)$, making independent decisions for distinct pairs $(i,j)$, where $i,j \in [n]$ and $i \neq j$. We refer to $\G(n,w)$ as a \textit{w-random \IN{graph}}. Importantly, for large enough samples, the \IN{cut distance} between a \IN{graphon} and a sample is small with high probability. 

It is often important to consider the average value of graphons over sets of the form $A \times B$. Let $w \in \W^1$ and let $A,B \subseteq [0,1]$. The \textit{\IN{cell average}} of $w$ over $A \times B$ is given by 
\begin{equation*}
\overline{w}(A\times B) = \frac{1}{|A \times B|}\iint_{A\times B}w~dxdy.
\end{equation*}
Given a fixed $L^p$ graphon $w \in \W^p$ and a partition $\mathcal{P}=\{S_1,\ldots,S_k\}$ of $[0,1]$ into nonempty measurable sets each of measure greater than 0, we define the step function $w_{\mathcal{P}}$ by
\begin{equation*}
    w_{\mathcal{P}}(x,y) = \frac{1}{|S_i\times S_j|}\iint_{S_i \times S_j}w(x,y)dxdy = \overline{w}(S_i\times S_j)\quad (x \in S_i,y \in S_j),
\end{equation*}
where the operator $w \mapsto w_{\mathcal{P}}$ is called the \textit{\IN{stepping operator}}. The \IN{stepping operator} does not increase either the $p$-norm or cut norm \cite{Borgs_2019} of any graphon $w$ it is applied to; indeed, if $w \in \W^1$, $\mathcal{P}$ is any partition of $[0,1]$ into a finite number of nonempty measurable sets, and $p \geq 1$, then
\begin{equation*}
    \|w_{\mathcal{P}}\|_p \leq \|w\|_p ~\text{ and } ~\|w_{\mathcal{P}}\|_{\square} \leq \|w\|_{\square}.
\end{equation*}
We note that this property is sometimes referred to as being \textit{contractive}.


\subsection{The Robinson property}
The Robinson property for matrices was first introduced in the study of the classical seriation problem \cite{robinson_1951}, whose objective is to order a set of items so that similar items are placed close to one another. A symmetric \IN{matrix} $A=[a_{ij}]$ is \emph{Robinson} if
\begin{equation*}
   i \leq j \leq k \implies a_{ik} \leq \min\{a_{ij},a_{jk}\},
\end{equation*}  
and is \textit{Robinsonian} if it becomes a Robinson \IN{matrix} after simultaneous application of a permutation $\pi$ to its rows and columns. In that case, the permutation $\pi$ is called a \textit{Robinson ordering} of $A$. Graphs $G$ with Robinsonian adjacency matrices $A_G$ are the well known \textit{unit interval graphs}, sometimes also referred to as $1$--\textit{dimensional geometric graphs}. A graph $G$ is a unit interval graph if and only if there exists an ordering $\prec$ on $V(G)$ such that 
\begin{equation*}
    \forall v,z,w \in V(G),~ v\prec z \prec w \text{ and } v \sim w \implies z\sim v \text{ and } z\sim w.
\end{equation*}
This ordering $\prec$ is precisely the Robinson ordering of $A_G$. Such Robinsonian matrices that are $\{0,1\}$-valued are also studied under the name of the ``consecutive 1s property'' (see \cite{fulkerson65}).

When measuring large quantities of matrices or graphs for the Robinson property, such as in quality analysis of data, it is often useful to determine if these objects can be viewed as samples of some process $w$. Were this the case, only the Robinson property of the process $w$ would need to be analyzed for quantitative results concerning the data as a whole. A graphon $w \in \W^p$ is said to be \emph{Robinson} if 
\begin{equation*}
   x \leq y \leq z \implies w(x,z) \leq \min\{w(x,y),w(y,z)\}.
\end{equation*}  
Robinson graphons were first considered in \cite{Chuangpishit_2015} under the name \textit{diagonally increasing graphons}. We call a \IN{graphon} \textit{Robinson almost everywhere}, or Robinson a.e.~for short, if it is equal a.e.~to a Robinson \IN{graphon}. A \IN{graphon} $w \in \W^p$ is called \textit{Robinsonian} if there exists a Robinson $L^p$ \IN{graphon} $u \in \W^p$ such that $\delta_{\Box}(w,u) = 0$. This natural extension of the Robinson property to the world of graphons results in the addition of significant complexity; thus, much effort has been focused on measuring and approximating the Robinson property---if a graphon is \textit{close enough} to being Robinson, its samples should share the same behavior.

\subsection{Sampling from Robinson graphons}
In \cite{Chuangpishit_2015}, the authors introduced a (labeled) graphon parameter \IN{$\Gamma$} to estimate the Robinsonicity of a given graphon. We recall its definition below.
\begin{definition}[The parameter \IN{$\Gamma$}]\label{def:Gamma}
For $w \in \W_0$ and a measurable subset $A$ of $[0,1]$, let
\begin{equation*}
    \Gamma(w,A) := \iint_{y<z}\left[\int_{x \in A \cap [0,y]} (w(x,z)-w(x,y))dx\right]_+dydz
    + \iint_{y<z}\left[\int_{x \in A \cap [z,1]} (w(x,y)-w(x,z))dx\right]_+dydz,
\end{equation*}
with $[x]_+ = \max(x,0)$. Moreover, let $\Gamma(w) := \sup_{A \in \mathcal{A}}\Gamma(w,A),$ where the supremum is taken over all  measurable subsets of $[0,1]$. If $G$ is a simple, labeled graph, then we define $\Gamma(G) := \Gamma(w_{G})$.
\end{definition}
It was shown in \cite{Chuangpishit_2015,ghandehari2023robust} that $\Gamma$ is suitable for Robinson measurement of graphons; that is, it possesses recognition, continuity, and stability. Thus, as \IN{$\Gamma$} is continuous with respect to \IN{graphon}s in the \IN{cut norm}, we hope that we can pass that continuity down to \IN{graph} sequences. Namely, given a convergent sequence of dense graphs $\{G_n\}$ and a \IN{graphon} $w \in \W_0$ where $\delta_{\Box}(G_n,w) \to 0$, we desire a result of the form $\Gamma(G_n) \to \Gamma(w)$. The authors showed a form of this independent of vertex labeling, which we recall below. 
\begin{proposition}[Corollary 6.5, \cite{Chuangpishit_2015}]
Let $\{G_n\}$ be a sequence of simple graphs with $|V(G_n)| \to \infty$ and let $w\in \W_0$ such that $\delta_{\Box}(G_n,w) \to 0$. If $\Phi_n$ is the set of all vertex labelings of $G_n$ and $\Phi$ is the set of all measure preserving bijections of $[0,1]$, then
\begin{equation*}
    \min_{\phi_n \in \Phi_n}\Gamma(G_n^{\phi_n}) \to \inf_{\psi \in \Phi}\Gamma(w^\psi).
\end{equation*}
\end{proposition}
In the case where the lefthand side of this limit tends to 0, it is tempting to conclude that the graphs $\{G_n\}$ are similar to random models sampled from a Robinsonian \IN{graphon}; however, this does not follow from any of the previous results. Indeed, while $\inf_{\psi \in \Phi} \Gamma(w^{\psi}) =0$, this would only imply that the $\{G_n\}$ are arbitrarily close in $\delta_{\Box}$ to Robinson graphons $\{u_n\}$---and a $\delta_{\Box}$--equivalence class may contain more than one Robinson graphon. The proof of stability of $\Gamma$ in \cite{ghandehari2023robust} combined with a functional analytic argument was necessary to show that graphs $\{G_n\}$ as mentioned above are indeed sampled from a unique Robinsonian graphon. 
\begin{theorem}[Theorem 5.3, \cite{ghandehari2023robust}]\label{thm:gammabigone}
Let $\{G_n\}$ be a sequence of simple graphs converging to a \IN{graphon} $w \in \W_0$ in $\delta_{\Box}$. Then $w$ is Robinsonian if and only if there exist vertex labelings $\{\phi_n\}_{n \geq 1}$ such that $\Gamma(G_n^{\phi_n}) \to 0$. 
\end{theorem}
%
\section{Sparse Robinsonian graph sequences}
In this section we state and prove our main result regarding the characterization of Robinsonian $L^p$ graphons. Specifically, we shall prove a generalization of Theorem~\ref{thm:gammabigone} where $w \in \mathcal{W}^p$ and where the simple graphs $\{G_n\}_{n \geq 1}$ that converge to $w$ do so by satisfying $\delta_{\Box}(G_n/\|G_n\|_1, w/\|w\|_1) \to 0.$ However, while $\Gamma$ is a suitable measurement of the Robinson property for graphons in the space $\W_0$, the techniques used to show continuity and stability of $\Gamma$ on $\W_0$ are unable to be extended to the more general space $\W^p$. Thus, we shall instead prove our characterization result in terms of the graphon parameter $\Lambda$, which was introduced in \cite{Ghandehari2023RobustRO} to address this discrepancy between $\W_0$ and $\W^p$. We recall its definition and relevant properties here. 
\begin{definition}[The parameter $\Lambda$]\label{def:lambda}
    Let $w \in \mathcal{W}^1$. Define
\begin{equation}\label{eq:lambda}
    \Lambda(w) = \frac{1}{2}\sup_{\substack{A \leq B \leq C, \\ |A|=|B|=|C|}}\bigg[\iint_{A\times C}w~dxdy-\iint_{B\times C} w~dxdy\bigg]\nonumber\\
    + \frac{1}{2}\sup_{\substack{X \leq Y \leq Z, \\ |X|=|Y|=|Z|}}\bigg[\iint_{X\times Z}w~dxdy-\iint_{X\times Y} w~dxdy\bigg]
\end{equation}
where $A,B,C$ and $X,Y,Z$ are measurable subsets of $[0,1]$ and $A \leq B$ if $a \leq b$ for all $a \in A$ and $b \in B$. If $G$ is a labeled graph, we define $\Lambda(G) := \Lambda(w_G)$.
\end{definition}
\begin{proposition}\label{prop:Lambda-prop}
Let $1\leq p\leq \infty$. Suppose $w, u\in \W^p$. Then we have
\begin{itemize}
\item[(i)] (Recognition) \IN{$\Lambda$} characterizes Robinson $L^p$-\IN{graphon}s, i.e.,  $w$ is Robinson a.e.~if and only if $\Lambda(w)=0$. \label{item:lambda:recognizes}
\item[(ii)] (Continuity) \IN{$\Lambda$} is continuous with respect to cut norm, i.e., $|\Lambda(w)-\Lambda(u)|\leq 2\|w-u\|_{\Box}$. \label{item:Lambda-cts}
\item[(iii)] (Stability) For every $w \in \W^p$ where $p > 5$, there exists a Robinson graphon $u \in \W^p$ satisfying
    \begin{equation*}
        \|u-w\|_{\Box} \leq 78\Lambda(w)^{\frac{p-5}{5p-5}}.
    \end{equation*} 
\end{itemize}
\end{proposition}
\noindent Informally, \IN{$\Lambda$} can be thought of as extending the Robinson property from pointwise comparison to comparison of blocks. We make special note of item $(iii)$, as this result is constructive---the Robinson graphon $u$ is the $\alpha$\textit{-Robinson approximation} of $w$ for some chosen $\alpha$, which is a generalized version of the Robinson construction introduced in \cite{ghandehari2023robust}.
\begin{definition}[$\alpha$-Robinson approximation for \IN{graphon}s]\label{def:R(w)} 
Let $p \geq 1$, and fix a parameter $0<\alpha<1$.
Given a \IN{graphon} $w\in {\mathcal W}^p$, the $\alpha$-\emph{\IN{Robinson approximation}} $R_w^{\alpha}$ of $w$ is the symmetric function defined as follows: For all $0 \leq x \leq y \leq 1$, define
\begin{equation*}
R_w^{\alpha}(x,y)=\sup\left\{ \frac{1}{|S \times T|}\iint_{S \times T} w\,:\, S\times T\subseteq [0,x]\times[y,1],\, |S|=|T|=\alpha\right\},
\end{equation*}
taking the convention that $\sup\emptyset =0$. Moreover, we set $R_w^{\alpha}=w$ if $\alpha=0$ and $w$ is Robinson.
\end{definition}
The goal is thus to show for $w$ and $\{G_n\}$ where $\delta_{\Box}(G_n/\|G_n\|_1,w/\|w\|_1)\to 0$ that $w$ is Robinsonian if and only if there exists a sequence of vertex labelings $\{\phi_n\}_{n \geq 1}$ such that $\Lambda(G_n^{\phi_n}/\|G_n\|_1) \to 0.$ While the forward direction is rather direct, the reverse direction requires some heavier machinery. There are two ways to show that $w$ is Robinsonian: Firstly, find a measure preserving bijection $\psi$ such that $\Lambda(w^{\psi})=0,$ and secondly, find a Robinson graphon $u \in \W^p$ such that $\delta_{\Box}(u/\|u\|_1,w/\|w\|_1) = 0$. 

Finding such a $\psi$ is not obvious---indeed, while it is tempting to think of $\psi$ as the limit object of the $\{\phi_n\}$, it is not clear that such an object need exist. Proposition~\ref{prop:Lambda-prop} (iii) provides an avenue of attack for the latter option of showing that $w$ is Robinsonian, as we can use these Robinson approximations of the associated graphons $w_{G_n}$ to find a Robinson graphon $u$ of cut distance 0 from $w$. We proceed with this plan below, beginning with a technical lemma showing that for certain partitions $\mathcal{P}$ of $[0,1]$, the step graphon $w_{\mathcal{P}}$ is Robinsonian.
\begin{lemma}\label{lem:robinsonianpartition}
    Let $\eta > 0$ and let $\mathcal{P} =\{P_1,\ldots,P_m\}$ be a partition of $[0,1]$ into measurable sets each of measure at least $\eta.$ If $w \in \W^p$ with $p > 5$ is a graphon such that there exists a sequence of graphs $\{G_n\}_{n \geq 1}$ and vertex labelings $\{\phi_n\}_{n \geq 1}$ where 
    $$\delta_{\Box}\left(\frac{G_n}{\|G_n\|_1},\frac{w}{\|w\|_1}\right) \to 0 \quad\quad\quad \text{and} \quad\quad\quad \Lambda\left(\frac{G_n^{\phi_n}}{\|G_n\|_1}\right) \to 0,$$ 
    then the step graphon $w_{\mathcal{P}}$ is Robinsonian.
\end{lemma}
\begin{proof}
Fix $\eta > 0$, let $\mathcal{P} =\{P_1,\ldots,P_m\}$ be a partition of $[0,1]$ into measurable sets each of measure at least $\eta$, and let $w_n := w_{G_n^{\phi_n}}/\|w_{G_n}\|_1$ and $w' := w/\|w\|_1$ for ease of reading. By the assumption of convergence of $w_n \to w'$ in cut distance, we have that $\delta_{\Box}((w_n)_{\mathcal{P}},(w')_{\mathcal{P}}) \to 0$. Furthermore, by Proposition~\ref{prop:Lambda-prop} (iii), for every $n \in \N$, there must exist a Robinson graphon $u_n \in \W^p$ such that 
\begin{equation*}
\|u_n - (w_n)_{\mathcal{P}}\|_{\Box} \leq 78\Lambda((w_n)_{\mathcal{P}})^{\frac{p-5}{5p-5}}.
\end{equation*}
We note that as $\Lambda(w_n) \to 0$, it must be the case that $\Lambda((w_n)_{\mathcal{P}}) \to0$ as well. This further implies that $\delta_{\Box}(u_n,(w')_{\mathcal{P}}) \to 0$ as $n \to \infty$, as
\begin{equation*}
    \delta_{\Box}(u_n,(w')_{\mathcal{P}}) \leq \delta_{\Box}(u_n,(w_n)_{\mathcal{P}}) + \delta_{\Box}((w_n)_{\mathcal{P}},(w')_{\mathcal{P}}) \leq \|u_n-(w_n)_{\mathcal{P}}\|_{\Box}+\delta_{\Box}((w_n)_{\mathcal{P}},(w')_{\mathcal{P}}).
\end{equation*}
Before we continue, we shall show that the functions $u_n$ are uniformly bounded above; indeed, note that
\begin{align*}
    |u_n(x,y)| = |R^{\alpha}_{(w_n)_{\mathcal{P}}}(x,y)| &= \sup_{\substack{A,B \subseteq [0,x]\times[y,1] \\ |A|=|B|=\alpha}}\left|\frac{1}{\alpha^2} \iint_{A\times B} (w_n)_{\mathcal{P}}\right|\\
    &\leq \|(w_n)_{\mathcal{P}}\|_{\infty} = \sup_{1\leq i,j \leq m}\left|\frac{1}{|P_i \times P_j|}\iint_{P_i \times P_j}w_n\right| \\
    &\leq \sup_{1\leq i,j \leq m}\frac{1}{|P_i \times P_j|}\|w_n\|_1\|\mathbbm{1}_{P_i \times P_j}\|_{\infty} \\
    &=\sup_{1\leq i,j \leq m}\frac{1}{|P_i \times P_j|} \leq \frac{1}{\eta^2}.
\end{align*}
Now, for all $n \in \N$, the functions $u_n$ are elements of the space $B_{1/\eta^2}(L^{\infty}[0,1]^2) = \{f \in L^{\infty}[0,1]^2:\|f\|_{\infty} \leq \frac{1}{\eta^2}\}$, and because the Banach space $L^{\infty}[0,1]^2$ is isometrically isomorphic to the Banach space dual of $L^{1}[0,1]^2$, it is possible to equip $B_{1/\eta^2}(L^{\infty}[0,1]^2)$ with the weak$^*$ topology induced by this duality. By the Banach-Alaglou theorem, $B_{1/\eta^2}(L^{\infty}[0,1]^2)$ is compact in this topology; additionally, as $L^1[0,1]^2$ is separable, $B_{1/\eta^2}(L^{\infty}[0,1]^2)$ is metrizable in the weak$^*$ topology and thus is sequentially compact as well. Therefore, $\{u_n\}_{n \in \N}$ has a weak$^*$ convergent subsequence.

Without loss of generality, we can thus assume that $\{u_n\}$ converges to some $z \in B_{1/\eta^2}(L^{\infty}[0,1]^2)$ in the weak$^*$ topology; i.e. for every $h \in L^1[0,1]^2$ we have that $\iint_{[0,1]^2}u_nh \to \iint_{[0,1]^2}zh$ as $n \to \infty.$ In particular, for all measurable sets $S,T \subseteq [0,1]$, we have that 
\begin{equation}\label{eq:robinsonconv}
    \iint_{S \times T} u_n \to \iint_{S \times T}z ~\text{ as } n \to \infty.
\end{equation}
By \eqref{eq:robinsonconv}, we get that for any partition $\mathcal{P}$ of $[0,1]$ whose elements are all bounded below in measure, $(u_n)_{\mathcal{P}}$ converges pointwise to $z_{\mathcal{P}}$. Furthermore, as every $(u_n)_{\mathcal{P}}$ is Robinson, it must be the case that $z_{\mathcal{P}}$ is Robinson as well. We can also show that $z$ is Robinson a.e. by letting $\mathcal{P}_n$ be any partition of $[0,1]$ into sets of exactly measure $1/n$ and noting that $\|z-z_{\mathcal{P}_n}\|_1 \to 0$. Additionally, again using \eqref{eq:robinsonconv}, it is true that $\delta_{\Box}(u_n,z) \to 0$ as $n \to \infty$. Therefore,  
\begin{equation*}
    \delta_{\Box}((w')_{\mathcal{P}},z) \leq \delta_{\Box}((w')_{\mathcal{P}},u_n)+\delta_{\Box}(u_n,z) \to 0
\end{equation*}
as $n \to \infty$, showing that $(w')_{\mathcal{P}}$ is Robinsonian and thus that $w_{\mathcal{P}}$ is Robinson as well, proving the claim.
\end{proof}
With this lemma at our disposal, we now state and prove our main result about the partial characterization of Robinsonian $L^p$ graphons. 
\begin{theorem}\label{thm:thebigone}
Let $\{G_n\}$ be a sequence of simple labeled graphs and let $w \in \W^p$ with $p > 5$ such that $\delta_{\Box}(G_n/\|G_n\|_1,w/\|w\|_1) \to 0$. Then, $w$ is Robinsonian if and only if $\Lambda(G^{\phi_n}_n/\|G_n\|_1) \to 0$ for some vertex labelings $\{\phi_n\}_{n \geq 1}$. 
\end{theorem}
\begin{proof}
We start with the forward direction; assume that $w$ is Robinsonian. Then there exists a Robinson graphon $u \in \W^p$ such that $\delta_{\Box}(w/\|w\|_1,u/\|u\|_1) = 0,$ from which it follows by using the triangle inequality that
\begin{equation*}
    \delta_{\Box}\left(\frac{G_n}{\|G_n\|_1},\frac{u}{\|u\|_1}\right) \leq \delta_{\Box}\left(\frac{G_n}{\|G_n\|_1},\frac{w}{\|w\|_1}\right) + \delta_{\Box}\left(\frac{w}{\|w\|_1},\frac{u}{\|u\|_1}\right) \to 0.
\end{equation*}
So there must exist a sequence of vertex labelings $\{\phi_n\}_{n \geq 1}$ such that $\|G_n^{\phi_n}/\|G_n\|_1 - u/\|u\|_1\|_{\Box} \to 0$ as $n$ tends to infinity. Noting that $\Lambda(u/\|u\|_1)=0$ because $u$ is Robinson, this implies that
\begin{equation*}
    \Lambda\left(\frac{G_n^{\phi_n}}{\|G_n\|_1}\right) = \left|\Lambda\left(\frac{G_n^{\phi_n}}{\|G_n\|_1}\right)-\Lambda\left(\frac{u}{\|u\|_1}\right)\right| \leq \left\|\frac{G_n^{\phi_n}}{\|G_n\|_1} - \frac{u}{\|u\|_1}\right\|_{\Box} \to 0
\end{equation*}
as required, proving the forward direction.

For the reverse direction, assume that each $G_n$ has a vertex labeling $\phi_n$ such that $\Lambda(G_n^{\phi_n}) \to 0$; note also that $\delta_{\Box}(G_n/\|G_n\|_1,w/\|w\|_1) \to 0$ by assumption. For all $n \in \N$, let $\mathcal{P}_n$ be any partition of $[0,1]$ into measurable sets each of measure exactly $1/n$, and consider the graphons $\{w_{\mathcal{P}_n}\}_{n \geq 1}.$ By Lemma \ref{lem:robinsonianpartition}, for all $n \in \N$, there exists a sequence $\{z_n\}_{n\in\N}$ of Robinson graphons and a sequence $\{\psi_n\}$ of measure preserving bijections of $[0,1]$ such that 
\begin{equation*}
    \|w_{\mathcal{P}_n}^{\psi_n} - z_n\|_{\Box} = 0.
\end{equation*}
We consider now the sequence of $L^p$ graphons $\{w^{\psi_n}_{\mathcal{P}_n}\}_{n \in \N}$. These are elements of a bounded ball of radius $\|w\|_p$ in the space $L^p[0,1]^2$, which is compact and sequentially compact in the weak$^*$ topology induced by $L^q[0,1]^2$. Thus, $\{w^{\psi_n}_{\mathcal{P}_n}\}_{n \in \N}$ has a weak$^*$ convergent subsequence in this topology and we can say without loss of generality that there exists some $y \in L^p[0,1]^2$ such that $\|y\|_p \leq \|w\|_p$ and that for all measurable $S,T \subseteq [0,1]$ we have
\begin{equation*}
    \iint_{S \times T} \left(w^{\psi_n}_{\mathcal{P}_n}\right)^p \to \iint_{S \times T}y^p ~\text{ as } n \to \infty.
\end{equation*}
As every $w^{\psi_n}_{\mathcal{P}_n}$ is Robinson a.e., it must be that $y$ is Robinson a.e. as well. We can therefore show that
\begin{equation*}
    \delta_{\Box}(w,y) \leq \delta_{\Box}(w,w_{\mathcal{P}_n}) + \delta_{\Box}(y,w_{\mathcal{P}_n}) \to 0
\end{equation*}
as $n\to \infty$, demonstrating that $w$ is Robinsonian and finishing the proof.
\end{proof}

\section{Further Directions}
Several natural questions present themselves as extensions to this work, the most immediate of which is whether the statement of Theorem~\ref{thm:thebigone} holds for $1 \leq p \leq 5$. As the requirement of $p > 5$ is an artifact of the proof method for Proposition~\ref{prop:Lambda-prop} (iii), an obvious method of attack for this question would be to improve the techniques used to show this result. It is additionally interesting to consider whether this characterization result could be shown for a more general concept of sparse graphons, such as the $k$-shapes introduced in \cite{KUNSZENTIKOVACS20191}, especially as these objects are defined independently of the vertex labeling of the graphs they are associated with. It is likely that a new parameter would need to be defined to quantify these approximation results in this extended setting, and it is of significant interest to the author to see whether similar claims could hold true for these more general objects. Finally, it is natural to wonder whether there are other graph properties for which such characterization results hold---monotonicity is a likely candidate due to its similarity to Robinsonicity, and categorizing other such properties seems a fruitful direction for future work.
\section{Acknowledgements}
The author would like to acknowledge Dr. Mahya Ghandehari, Charli Klein, and Sarah Tillman for their help in editing this paper and the Department of Mathematics at Toronto Metropolitan University for its continued support throughout the process of this research. This research did not receive any specific grant from funding agencies in the public, commercial, or not-for-profit sectors.
\bibliography{bibliography}

\begin{thebibliography}{16}
\providecommand{\natexlab}[1]{#1}
\providecommand{\url}[1]{\texttt{#1}}
\expandafter\ifx\csname urlstyle\endcsname\relax
  \providecommand{\doi}[1]{doi: #1}\else
  \providecommand{\doi}{doi: \begingroup \urlstyle{rm}\Url}\fi

\bibitem[Borgs et~al.(2018)Borgs, Chayes, Cohn, and Zhao]{Borgs_2018}
C.~Borgs, J.~Chayes, H.~Cohn, and Y.~Zhao.
\newblock {An $L^{p}$ theory of sparse graph convergence II: LD convergence,
  quotients and right convergence}.
\newblock \emph{The Annals of Probability}, 46\penalty0 (1), Jan 2018.
\newblock ISSN 0091-1798.

\bibitem[Borgs et~al.(2019)Borgs, Chayes, Cohn, and Zhao]{Borgs_2019}
C.~Borgs, J.~Chayes, H.~Cohn, and Y.~Zhao.
\newblock {An $L^p$ theory of sparse graph convergence I: Limits, sparse random
  graph models, and power law distributions}.
\newblock \emph{Transactions of the American Mathematical Society},
  372\penalty0 (5):\penalty0 3019–3062, May 2019.
\newblock ISSN 1088-6850.

\bibitem[Chepoi and Seston(2009)]{chepoi2009}
V.~Chepoi and M.~Seston.
\newblock {Seriation in the Presence of Errors: A Factor 16 Approximation
  Algorithm for $l_{\infty}$-Fitting Robinson Structures to Distances}.
\newblock \emph{Algorithmica}, 59:\penalty0 521--568, 02 2009.

\bibitem[Chuangpishit et~al.(2015)Chuangpishit, Ghandehari, Hurshman, Janssen,
  and Kalyaniwalla]{Chuangpishit_2015}
H.~Chuangpishit, M.~Ghandehari, M.~Hurshman, J.~Janssen, and N.~Kalyaniwalla.
\newblock Linear embeddings of graphs and graph limits.
\newblock \emph{Journal of Combinatorial Theory, Series B}, 113:\penalty0
  162–184, Jul 2015.
\newblock ISSN 0095-8956.

\bibitem[Flammarion et~al.(2019)Flammarion, Mao, and
  Rigollet]{flammarion2016optimal}
N.~Flammarion, C.~Mao, and P.~Rigollet.
\newblock {Optimal rates of statistical seriation}.
\newblock \emph{Bernoulli}, 25\penalty0 (1):\penalty0 623--653, 2019.

\bibitem[Fogel et~al.(2016)Fogel, d'Aspremont, and Vojnovic]{fogel2014}
F.~Fogel, A.~d'Aspremont, and M.~Vojnovic.
\newblock Spectral ranking using seriation.
\newblock \emph{J. Mach. Learn. Res.}, 17\penalty0 (1):\penalty0 3013–3057,
  January 2016.
\newblock ISSN 1532-4435.

\bibitem[Frieze and Kannan(1999)]{friezekannan}
A.~Frieze and R.~Kannan.
\newblock Quick approximation to matrices and applications.
\newblock \emph{Combinatorica}, 19:\penalty0 175--220, 02 1999.

\bibitem[Fulkerson and Gross(1965)]{fulkerson65}
D.~R. Fulkerson and O.~A. Gross.
\newblock Incidence matrices and interval graphs.
\newblock \emph{Pacific J. Math.}, 15:\penalty0 835--855, 1965.
\newblock ISSN 0030-8730,1945-5844.

\bibitem[Ghandehari and J.(2024)]{ghandehari2023robust}
M.~Ghandehari and Jeannette J.
\newblock Graph sequences sampled from {R}obinson graphons.
\newblock \emph{European J. Combin.}, 116:\penalty0 Paper No. 103859, 31, 2024.
\newblock ISSN 0195-6698,1095-9971.

\bibitem[Ghandehari and Mishura(2024)]{Ghandehari2023RobustRO}
M.~Ghandehari and T.~Mishura.
\newblock {Robust Recovery of Robinson Property in $L^p$-Graphons: A Cut-Norm
  Approach}.
\newblock \emph{Electron. J. Comb.}, 31, 2024.

\bibitem[Hladký and Rocha(2020)]{HLADKY2020103108}
J.~Hladký and I.~Rocha.
\newblock Independent sets, cliques, and colorings in graphons.
\newblock \emph{European Journal of Combinatorics}, 88:\penalty0 103108, 2020.
\newblock ISSN 0195-6698.
\newblock Selected papers of EuroComb17.

\bibitem[Kunszenti-Kovács et~al.(2019)Kunszenti-Kovács, Lovász, and
  Szegedy]{KUNSZENTIKOVACS20191}
D.~Kunszenti-Kovács, L.~Lovász, and B.~Szegedy.
\newblock Measures on the square as sparse graph limits.
\newblock \emph{Journal of Combinatorial Theory, Series B}, 138:\penalty0
  1--40, 2019.
\newblock ISSN 0095-8956.

\bibitem[Laurent et~al.(2017)Laurent, Seminaroti, and Tanigawa]{laurent2017}
M.~Laurent, M.~Seminaroti, and S.~Tanigawa.
\newblock A structural characterization for certifying robinsonian matrices.
\newblock \emph{Electronic Journal of Combinatorics}, 24, 01 2017.

\bibitem[Liiv(2010)]{Liiv_2010}
I.~Liiv.
\newblock Seriation and matrix reordering methods: An historical overview.
\newblock \emph{Statistical Analysis and Data Mining: The ASA Data Science
  Journal}, 3\penalty0 (2):\penalty0 70--91, 2010.

\bibitem[Lovász and Szegedy(2006)]{LOVASZ2006933}
L.~Lovász and B.~Szegedy.
\newblock Limits of dense graph sequences.
\newblock \emph{Journal of Combinatorial Theory, Series B}, 96\penalty0
  (6):\penalty0 933--957, 2006.
\newblock ISSN 0095-8956.

\bibitem[Robinson(1951)]{robinson_1951}
W.~S. Robinson.
\newblock A method for chronologically ordering archaeological deposits.
\newblock \emph{American Antiquity}, 16\penalty0 (4):\penalty0 293–301, 1951.

\end{thebibliography}
\end{document}